\documentclass[12pt]{article}
\title{A short note on model theory of $\C((t))$}
\author{Zhentao Zhang}

\date{}

\usepackage{amsmath, amssymb, amsthm}    	
\usepackage{fullpage} 	
\usepackage{amscd}
\usepackage{hyperref}
\usepackage[all]{xy}
\usepackage{centernot}
\usepackage{color,xcolor} 
\usepackage{ulem}

\DeclareMathOperator*{\forkindep}{\raise0.2ex\hbox{\ooalign{\hidewidth$\vert$\hidewidth\cr\raise-0.9ex\hbox{$\smile$}}}}

\newcommand{\Int}{\operatorname{Int}}

\newcommand{\res}{\operatorname{res}}

\newcommand{\id}{\operatorname{id}}

\newcommand{\acl}{\operatorname{acl}}

\newcommand{\cl}{\operatorname{cl}}

\newcommand{\F}{\mathcal{F}}
\newcommand{\N}{\mathbb{N}}
\newcommand{\C}{\mathbb{C}}

\newcommand{\Z}{\mathbb{Z}}
\newcommand{\Q}{\mathbb{Q}}

\newcommand{\pCF}{p\mathrm{CF}}
\newcommand{\ACF}{\mathrm{ACF}}

\newtheorem{theorem}{Theorem}[section] 
\newtheorem{lemma}[theorem]{Lemma}

\newtheorem{conj}[theorem]{Conjecture}

\newtheorem{prop}[theorem]{Proposition}
\newtheorem{proposition-eh}[theorem]{Proposition(?)}
\newtheorem*{theorem-star}{Main Theorem}
\newtheorem*{theorem-star-A}{Theorem A}
\newtheorem*{theorem-star-B}{Theorem B}
\newtheorem*{conjecture-star}{Conjecture}
\newtheorem*{lemma-star}{Lemma}
\newtheorem*{claim-star}{Claim}

\theoremstyle{definition}

\newtheorem{remark}[theorem]{Remark}

\newcommand{\PP}{\mathbb{P}}
\newcommand{\trans}{\mathrm{trans}}
\newcommand{\M}{\mathbb{M}}
\newcommand{\cO}{\mathcal{O}}

\newcommand{\m}{\mathfrak{m}}

\begin{document}
\maketitle
\begin{abstract}
In this short note, we study $\C((t))$ in the language of valued rings. We show that a definable subset of $\C((t))^n$ (or in monster model, $\M^n$) is definably compact iff it is closed and unbounded. Then we give some comments on definable groups over $\C((t))$.  
\end{abstract}

\section{Introduction}

In recent years, the study on definable groups in $\pCF$ has been fruitful. In particular, dfg groups and fsg groups play an important role in studying definable groups. In \cite{LY}, C. Ling and N. Yao gave a glance on the parallel theory on $\C((t))$ (in the language of valued rings), specially, for ``bounded'' definable groups. But the final conjecture in \cite{LY} was not given properly, because they did not give ``boundedness'' properly and one can easily embed the additive group of $\C((t))$ into $\C[[t]]$ definably.

In Section \ref{sect-def-cpt}, we study the things in the affine space $\C((t))^n$ (with the usual topology). We show that a definable subset of $\C((t))^n$ is definably compact iff it is closed and unbounded. 

In Section \ref{sect-group}, we have some comments on definable groups. We can equip definable groups with a group topology by finitely many affine pieces and give a $\Gamma$-exhaustion as in \cite{JY-1}. This result gives a foundation for further research.  
In particular, we can give a proper version of the final conjecture of \cite{LY}.

In Section \ref{sect-app}, we will give some results on definable groups, for an example, the Peterzil–Steinhorn theorem for abelian groups.

For notation, we always study $\C((t))$ in $L$, the language of rings, and let $\M$ be a monster model. We denote the valued ring of $\C((t))$ by $\cO$, the maximal ideal of $\cO$ by $\m$ and $\Gamma$. On arbitrary model $M$, we denote the corresponding ones by $\cO_M$ and $\m_M$. We denote the map of $x\mapsto x/\m_\M$ for $x\in \cO_\M$ by $\res$.
Note that $k_M:=\res(\cO_M)$ is the residue field of $M$ which is algebraically closed. 

\section*{Acknowledgements}
Thanks to Yang Yang and Ningyuan Yao for checking some details and beneficial discussions. 

\section{Definable compactness}\label{sect-def-cpt}

One can find the definition and basic properties of definable compactness in \cite{JY-1} Definition 2.1, Fact 2.2 and Remark 2.3. We first show that $\cO$ is definably compact. Before that we recall some easy results. Firstly, for every definable subset $X$ of $\cO$, $\res(X)$ is finite or cofinite in $\C$ and therefore definable in $\ACF_0$ (Theorem 1 \cite{LY}). Secondly, from \cite{LY} Section 3, on each model $M$, there is a unique type $p_{\trans,M}\in S_\cO(M)$ such that $\res(p_{\trans,M})$ is the  transcendental type over $k_M$ in $\ACF_0$. Moreover, $p_{\trans, \M}$ is generically stable and both dfg and fsg type for $(\cO,+)$. Thirdly, we have

\begin{lemma}[Without any assumption for theories]
Let $\F$ be a definable family in $M$ with F.I.P. (finite intersection property) and $\F\subset p\in S(M)$.  Let $M\prec N$ and $\bar p\in S(N)$ be an heir of $p$. Then $\F(N)\subset \bar p$. 
\end{lemma}
\begin{proof}
Assume that $\F=\{\varphi(M;b):b\in M, M\models\theta(b)\}$ for some $\varphi(x;y), \theta(y)\in L_M$. Let $\chi(x;y)$ be $(\neg\varphi(x;y))\wedge \theta(y)$. If $\F(N)$ is not contained in $\bar p$, then there is $b^*\in N$ such that $\chi(x;b^*)\in \bar p$. As $\bar p$ is heir over $M$, there is $b'\in M$ such that $\chi(x;b')\in p$ which contradicts to $\varphi(x;b')\in p$.
\end{proof}

Then we have 
\begin{lemma}
$\C[[t]]$ is definably compact.  
\end{lemma}
\begin{proof}
Let $\F$ be a definable directed system of closed subset of $\C[[t]]$. Then $\bar\F:=\{X/\m:X\in \F\}$ is a collection of definable subsets of $\C$ (in $\ACF_0)$) with F.I.P. 

If $\bar\F$ is contained in the transcendental type over $\C$ (in $\ACF_0$), then $\F$ is contained in $p=p_{\trans,\C((t))}$, the generic type of $(\cO,+)$ over $\C((t))$. Let $\bar p$ be the global heir of $p$. Clearly, $\res(\bar p)$ is the transcendental type over the residue field of $\M$ and $\bar p=p_{\trans,\M}$.
As $\bar p$ is heir over $\C((t))$, we have $\F(\M)\subset \bar p$. 
Then it is easy to see that each $X\in\F(\M)$ there is a finite subset $\Delta_X$ of $\res(\M)$ such that 
 $X\supset \cO_\M\backslash (\bigcup_{c\in \Delta_X} c+\m_\M)$.
Note that for each $n$, $|\Delta_X|\leq n$'' can be written as 
$$\forall x_1\in\cO_\M,\dots,\forall x_{n+1}\in \cO_\M(\bigwedge_{ j\neq k} \res(x_j)\neq \res(x_k))\rightarrow
\bigvee_{i=1}^{n+1}(X\cap (x_i+\m_\M)=x_i+\m_\M)$$
Let $D_n:=\{X\in \F(\M): |\Delta_X|\leq n\}$. Then $D_n$ is a definable family and $\F(\M)=\bigcup_n D_n$. By compactness, $\F(\M)=D_N$ for some $N\in\N$. 
Then $|\Delta_X|\leq N$ for every $X\in\F$. As $\bar\F$ is directed, we have that $\bigcap \F\neq \emptyset$.

If $\bar\F$ is not contained in the transcendental type over $\C$, then there is $a_0\in \C$, such that $a_0\in \bigcap \bar\F$.  Then we let $\F_1=\{t^{-1}(\m \cap (X-a_0)):X\in \F\}$. If $\bar \F_1$ is contained in the transcendental type over $\C$, we have done because $x\mapsto tx$ is an isomorphism of additive groups of $\C[[t]]$ to $\m$. If not, there is $a_1\in \bigcap \bar \F_1$ and we let $\F_2==\{t^{-1}(\m \cap (X-a_1):X\in \F_1\}$

Inductively, if in some step, $\bar \F_i$ is contained in the transcendental type over $\C$, then we have done. If not, we can find $a_{i}\in \bigcap \bar \F_i$ and let $\F_{i+1}==\{t^{-1}(\m \cap (X-a_i)):X\in \F_i\}$. Then we will have a sequence $(a_i)_i$ of $\C$. Let $a=\sum_i a_i t^i$. As each $X\in\F$ is closed, we have that $a\in \bigcap \F$.
\end{proof}

\begin{theorem}\label{def-cpt-aff}
A definable set $X$ in $\C((t))^n$ (or $\M^n$) is definably compact iff it is closed and bounded.
\end{theorem}
\begin{proof}
By Remark 2.3 \cite{JY-1}, it suffices to deal with the global version which will meet more parameters. Note that the direction of $\Rightarrow$ is trivial. For the other direction, we note that the maps of $x\mapsto a+cx$ is a homeomorphism of $\cO_\M\rightarrow a+c\cO_\M$ for every $a\in\M$ and $c\in\M^*$. Then apply the basic properties on homeomorphisms, closed subspaces and  Cartesian products.
\end{proof}

\begin{remark}\label{flag}
For some examples, we have that the $n$-dimensional projective space $\PP^n$ and the space of complete $n$-flags are definably compact (see Fact 2.2 \cite{JY-1}).
\end{remark}

\section{Topology on definable groups}\label{sect-group}

In this section, we work in $\M$ and let $M$ be a small model. Note that some comments are basic on \cite{PS} which also work for other dp-mimial fields.

Note that for two definable sets $Y\subset X$, we say that $Y$ is  \textbf{large} in $X$ if $\dim X\backslash Y<\dim X$. We use the concepts of \textbf{correspondence} ($f: X\rightrightarrows Y$) and its 
continuity from \cite{PS}. The correspondence $f$ is called an \textbf{$e$-(resp. $e_\leq$-)correspondence} when $|f(x)|=e$(resp. $\leq e$) for each $x\in X$. It is easy from compactness that every correspondence is an $e_\leq$-correspondence for some $e\in\N$. Moreover, we have 

\begin{lemma}\label{e-corres}
Let $U\subset \M^n$ be definable and  open and $f: U\rightrightarrows \M^l$ be a correspondence. Then there are finitely many pairwise disjoint $U_i\subset_\text{open} U$ and natural numbers $e_i$ such that $f|_{U_i}$ is an $e_i$-correspondence and $\biguplus_i U_i$ is large in $U$.
\end{lemma}
\begin{proof}
Let $f$ is an $e_\leq$-correspondence. Induction on $e$. When $e=1$, it is clear. Assume that $e>1$. We let $X=\{x\in U: |f(x)|=e\}$, $V=\Int X$ and $W=U\backslash \cl(X)$. If $\dim W<\dim U$, then $V\subset_\text{open} U$ is large and it is done. Now assume that $\dim W=\dim U$.
By induction, we have wanted $U_i$ for $f|_{W}$. Then these $U_i$ together with $V$ are what we want for $f$.  
\end{proof}

Then we have a manifold-like structure of finite type.

\begin{lemma}\label{key-lemma}
Let $X\subset \M^n$ be an $M$-definable set and $f: X\rightarrow \M$ be an $M$-definable function. We can find $O$ an $M$-definable large subset of $X$, and a finite partition of $O$ by $M$-definable subsets $O_i$, natural numbers $e_i,r_i$, $M$-definable $e_i$-to-one continuous and open functions $\pi_i: O_i\rightarrow \pi(O_i)=U_i\subset_\text{open} \M^n$ and $M$-definable continuous $r_i$-correspondences $f_i: U_i\rightrightarrows\M$ such that $f_i\circ\pi_i=f$ on $O_i$.
\end{lemma}

\begin{proof}
Let $X\subset \M^n$ be an $M$-definable set and $f: X\rightarrow \M^l$ be an $M$-definable function. For every $a\in X$, there is a partial coordinate projection $\pi_a:(x_1,\dots,x_n)\mapsto (x_{i_1},\dots,x_{i_d})$ such that the coordinate components of $\pi(a)$ are independent over $M$ and $a\in \acl(M,\pi(a))$.\footnote{If $a\in \acl(M)$, we let $\pi_a=0: \M^n\rightarrow \{0\}$.}
Then there is a $L_M$-formula $\varphi(x;y)$ such that $\varphi(a;\pi_a(a))$ and $e_a:=|\varphi(\M; \pi_a(a))|\in\N$. 
Then we let $$\psi_a(x):=\varphi(x;\pi_a(x))\wedge\text{``}|\varphi(\M;\pi(x))|=e_a\text{''}$$
Then $(\psi_a(\M))_{a\in X}$ covers $X$. As every $\psi_a$ is in $L_M$, the cover has a small index. By compactness, there is a finite cover $(\psi_{a_i}(\M))_i$. 
We denote $\psi_{a_i}$ by $\psi_i$ and we can modify these $\psi_i$ to ensure they pairwise disjoint. We denote $\varphi_{a_i}$ by $\varphi_i$. We denote $\pi_{a_i}$ by $\pi_i: \M^n\rightarrow \M^{d_i}$. We denote $e_{a_i}$ by $e_i$. It is clear that each $\pi_i$ is continuous, open and $(e_i)_\leq$-to-one\footnote{It means that at most $e_i$ many points have the same image.} on $\psi_i(\M)$.

Assume that $\dim X=d$. Then $\max_i d_i=d$. Note that for $a\in \psi_i(\M)$, $\dim(a/M)\leq d_i$.
Let $I=\{i:d_i=d\}$. Then for $i\in I$, we let $V_i=\Int(\pi_i(\psi_i(\M)))\subset \M^d$. It is clear that $\pi^{-1}:V\rightrightarrows X$ is a continuous correspondence. By Lemma \ref{e-corres}, we can assume that, without loss of generality (By partitioning $V_i$ into more pieces if necessary), there is $W_i\subset_\text{open} V_i$ on which $\pi_i^{-1}$ is an $e_i$-correspondence.
Then $f_i: W_i\rightrightarrows \M^l: u\mapsto \{f(x):x\in W_i\cap f^{-1}(u)\} $ is also a correspondence definable over $M$. By By Lemma \ref{e-corres}, we also assume that, without loss of generality (by taking new $W_i$), $f_i$ is an $r_i$-correspondence. Then by  Proposition 3.7 \cite{PS}, there is an $M$-definable open and large subset $U_i$ of $W_i$ such that $f_i$ is continuous on $U_i$. Let $O_i=\pi^{-1}(U_i)$ for $i\in I$ and $O=\biguplus_{i\in I}O_i$. Then it is easy to see that $O$ is large in $X$.

The above $O,O_i,U_i,e_i, r_i$ and $f_i$ are what we need.
\end{proof}

\begin{remark}\label{mfd}
Let notations be as above. It is clear that $f$ is continuous on $O$, a large subset of $X$. We call the structure given by those data on $O$ a \textbf{weak manifold}.

Let $a\in O_i$ and $u=\pi_i(a)$. As $\pi_i^{-1}: U_i\rightrightarrows O_i$ is a continuous $e_i$-correspondence, by Lemma 3.2 \cite{PS}, there is an open neighborhood $V_u$ of $u$ on which
$\pi_i^{-1}$ is split into definable functions $g_1,\dots,g_{e_i}$.\footnote{They can be defined over $\acl(M,u)$ and possibly can not defined over $M$.} Then there is a unique $1\leq j\leq e_i$ such that $\theta_a=g_j$ and  $a\in \theta_a(V_u):=W_a$. It is clear that $\pi_i\circ \theta_a=\id_{W_a}$ and $\theta_a\circ \pi_i=\id_{V_u}$. Hence, $\theta_a^{-1}:W_a\cong V_u$ gives local coordinates around $a$.
Let $a, b\in O$. If $W_a\cap W_{b}\neq \emptyset$, then $a,b$ are in the same $O_i$. The transfer map $$V_{\pi_i(a)}\supset \pi_i(W_{a}\cap W_{b})\stackrel{\theta_a}\longrightarrow W_a\cap W_b\stackrel{\theta_b^{-1}(=\pi_i)}\longrightarrow \pi_i(W_a\cap W_b)\subset V_{\pi_i(b)}$$
is obviously continuous. Hence, we can give a structure of manifold (but with infinitely many affine pieces) on $O$, a large subset of $X$.
\end{remark}

Modifying the way how A. Pillay gave a definable group a Lie structure, we have 

\begin{theorem}
Every definable group over $M$ has a \textbf{weak Lie structure} over $M$, namely, a weak manifold structure over $M$ on which the group operations are continuous. 
\end{theorem}

Note that a weak manifold is divided into finitely many pieces and each piece is mapped to an open subset of the affine space by a finite-to-one, continuous and open map. So the definable compactness on a weak manifold can be determined by those pieces and  determined by those corresponding affine open sets. Also, via those pieces, we can give a $\Gamma$-exhaustion (see Definition 2.6 \cite{JY-1}). Hence, we have

\begin{prop}
Every definable weak manifold admits a $\Gamma$-exhaustion.
\end{prop}

Then on a definable weak manifold,
we can talk about boundedness via a $\Gamma$-exhaustion.
It is easy from Theorem \ref{def-cpt-aff} that

\begin{prop}
A definable subset of a definable weak manifold is definably compact iff it is closed and bounded.
\end{prop}

Moreover, for definable weak Lie group, as in \cite{JY-1}, we have
\begin{prop}
Every definable group admits a good neighborhood basis (see Definition 2.27 \cite{JY-1} and for convenience, we can add Proposition 2.29 \cite{JY-1} into the definition).
\end{prop}

By a similar proof of Proposition 3.1 \cite{J}, we have

\begin{prop}
Let $G$ be a fsg group. Then $G$ is definably compact. 
\end{prop}
For the other direction on $\C((t))$, we have the revised conjecture of \cite{LY} that
\begin{conj}
Let $G$ be a definably compact group defined over $\C((t))$. Then
\begin{enumerate}
    \item  $G$ is generically stable;
    \item  $G^{00}=G^0$ is a finite index subgroup of $G$. 
\end{enumerate}
\end{conj}

\section{More results}\label{sect-app}

Firstly, we can modify the proof of \cite{JY-1} to show that

\begin{theorem}
Let $G$ be an abelian group definable over $\C((t))$. If $G$ is not definably compact, then it has a one-dimensional unbounded subgroup. 
\end{theorem}

There are two key points. One point is the bad gap configuration theorem (see Section 4.2 \cite{JY-1}) which holds for dp-finite fields. Another point is that in Lemma 5.8 \cite{JY-1}, the intersection has a bounded index which is also true in our settings. Now, the only thing we need to explain are $\mu$ and the standard part map that occurred in the proof. 

We first explain the standard part map on affine space and especially on $\M^1$. Note that $1$-types over $\C((t))$ are clear (see \cite{Type}). The difference to $\Q_p$ is that there exist types extending the residue field.
Note that every $1$-type over $\C((t))$ is definable. We let $\nu$, the $\mu$ for $\M^1$, to be the collection of realizations of the types containing $v(x)>n$ for every $n\in \Z=\Gamma_{\C((t))}$. Let $\mu$ on $\M^n$ to be $\nu^n$. 
Let $G$ be a definable group over $\C((t))$. We give $\mu_G$ by the weak Lie structure. As $1_G\in G(\C((t)))$, by Remark \ref{mfd} or Lemma 3.2 \cite{PS}, we have a neighborhood $U$ of $1_G$ defined over $\C((t))$ such the correspondence on the piece containing $1_G$ splits. Hence, $U$ is an affine open subset containing $1_G$ and we give $\mu_G$ on $U$. It is easy from continuity that $g^{-1} \mu_G g=\mu_G$ for $g\in G(\C((t)))$.  Let $V_G=\mu_G G(\C((t)))$. Then $\mu_G$ is normal subgroup of $V_G$. We let the standard part map to be the natural projection of $V_G\rightarrow V_G/\mu_G=G(\C((t)))$. 

Then follow the proof in \cite{JY-1}, we have the above Peterzil–Steinhorn theorem on abelian definable groups proved.

Secondly, while $\C((t))$ is not locally compact, but it is ``locally definably compact'' in our settings. Then we have something originally on locally compact fields (Theorem 3.1 \cite{lcpt}). 

\begin{prop}
Let $G$ be an algebraic group defined over $\C((t))$, and let $H$ be a algebraic subgroup
of $G$ over $\C((t))$.\footnote{For convenience, we write $G=G(\C((t)))$.} The quotient $G/H$\footnote{It is a geometric quotient and hence, definable.} is definably compact if and only if H contains a
maximal connected $\C((t))$-split solvable subgroup of the connected component (as an algebraic group)
of $G$. 
\end{prop}

One can check the original proof in \cite{lcpt}. Note that the space of complete $n$-flags is definably compact (Remark \ref{flag}) and the implicit function theorem is from the henselian property. 

Then we have some conjectures

\begin{conj}
Let $G$ be a group defined in $\C((t))$. Then $G^0=G^{00}$. Consequently (like Corollary 4.3 \cite{JY-2}), $G$ is virtually an open subgroup of an algebraic group over $\C((t))$.
\end{conj}

\begin{conj}
Let $G$ be a dfg group defined in $\C((t))$. Then $G$ is virtually a $\C((t))$-split solvable algebraic group.
\end{conj}

\begin{conj}
Let $G$ be a group defined in $\C((t))$. Consider everything in virtual sense.
Then there is $H$, a $\C((t))$-definable dfg subgroup of $G$, such that $G/H$ is definable (not only interpretable) and definably compact. 
\end{conj}

These conjectures can be also asked in other models. Note that the first conjecture is not true in general.


\begin{thebibliography}{99}

\bibitem{LY}
C. Ling and N. Yao, Generic stability of linear algebraic groups over $\C((t))$, arXiv:2307.05546 [math.LO].

\bibitem{JY-1} W. Johnson and N. Yao, On non-compact $p$-adic definable groups, The Journal of Symbolic Logic,
Volume 87, Number 1, March 2022.



\bibitem{PS} P. Simon and E. Walsberg, Tame Topology over dp-Minimal Structures,
Notre Dame J. Formal Logic 60(1): 61-76 (2019). 


\bibitem{J} W. Johnson, A note on 
fsg groups in $p$-adically closed fields, Mathematical Logic Quarterly, Volume69, Issue1,
February 2023, Pages 50-57.


\bibitem{Type}  F. Delon, Types sur $C((x))$, Groupe d’\'{e}tude de th´\'{e}ories stables,
Volume 2, 1978-1979.


\bibitem{lcpt}  V. Platonov, A. Rapinchuk and I. Rapinchuk, Algebraic Groups and Number Theory, 2nd edition, Cambridge University Press, 2023. 

\bibitem{JY-2} 
W. Johnson and N. Yao, 
Abelian groups definable in $p$-adically closed fields, The Journal of Symbolic Logic, First View, pp. 1 - 22.

\end{thebibliography}
\end{document}